\documentclass[12pt]{article}
\usepackage{amsmath}
\usepackage{amsfonts}
\usepackage{amssymb}
\usepackage{amscd}
\usepackage{amsthm}
\usepackage{bbm}
\usepackage{color}      
\usepackage[linktocpage=true,plainpages=false,pdfpagelabels=false]{hyperref}

\input xypic

\usepackage[matrix,arrow,curve]{xy}

\newcommand{\Z}{{\mathbb Z}}

\newcommand{\N}{{\mathbb N}}
\newcommand{\Q}{{\mathbb Q}}

\newcommand{\Ec}{{\mathcal E}}

\newcommand{\res}{\mathop{\rm res}}       

\newcommand{\Kc}{{\mathcal K}}

\newcommand{\LL}{\mathcal L}

\newcommand{\OO}{{\mathcal O}}

\newcommand{\End}{{\rm End}}

\newcommand{\Spec}{{\rm Spec} }

\newcommand{\sgn}{{\rm sgn}}

\newcommand{\gm}{\mathbb G_{m}}

\newcommand{\z}{{\mathbb{Z}}}

\newcommand{\lrto}{\longrightarrow}

\newcommand{\lo}{L}
\newcommand{\uz}{\underline{\mathbb{Z}}}

\newcommand{\vv}{{\mathbb{V}}}

\theoremstyle{plain}
\newtheorem{theor}{Theorem}[section]
\newtheorem{prop}[theor]{Proposition}

\newtheorem{corol}[theor]{Corollary}
\newtheorem{lemma}[theor]{Lemma}

\theoremstyle{remark}
\newtheorem{rmk}[theor]{Remark}
\newtheorem{examp}[theor]{Example}

\theoremstyle{definition}
\newtheorem{defin}[theor]{Definition}
\newtheorem{defin-prop}[theor]{Definition-Proposition}

\newcommand{\quash}[1]{}

\title{Higher-dimensional Contou-Carr\`ere symbol and continuous automorphisms\footnotetext{This work is supported by the RSF under a grant 14-50-00005.}}
\author{Sergey Gorchinskiy and Denis Osipov}
\date{}

\begin{document}

\maketitle

\begin{abstract}
We prove that the higher-dimensional Contou-Carr\`ere symbol is invariant under continuous automorphisms of algebras of iterated Laurent series over a ring.
Applying this property, we obtain a new explicit formula for the higher-dimensional Contou-Carr\`ere symbol. Unlike previously known formulas, this formula is given over an arbitrary ring, not necessarily a $\Q$-algebra, and does not involve algebraic $K$-theory.
\end{abstract}

\tableofcontents

\section{Introduction}

The goal of this paper is to study a relation between the higher-dimensional Contou-Carr\`ere symbol and continous endomorphisms of the algebra of iterated Laurent series over a ring.

\medskip

More precisely, given a natural $n\geqslant 1$, an algebra of iterated Laurent series over a ring~$A$ is the algebra $A((t_1)) \ldots ((t_n))$, which we denote, for short, by $\LL^n(A)$ (see Section~\ref{sec:prel} for more details). The $n$-dimensional Contou-Carr\`ere symbol is an antisymmetric multilinear map
\begin{equation}\label{eq:CCintr}
CC_n\,:\,\big(\LL^n(A)^*\big)^{\times (n+1)} \longrightarrow A^*\,,
\end{equation}
which is functorial with respect to a ring $A$, and where for a ring $R$ we denote by $R^*$  the group of its invertible elements. The one-dimensional Contou-Carr\`ere symbol (or, simply, the Contou-Carr\`ere symbol) was first introduced and studied by Contou-Carr\`ere himself, see~\cite{CC1,CC2}, and by Deligne, see~\cite{Del}. The two-dimensional Contou-Carr\`ere symbol was introduced and studied by the second named author and Zhu in~\cite{OZ1}. For arbitrary~$n$ (essentially, for $n > 2$) the higher-dimensional Contou-Carr\`ere symbol was extensively studied by the authors in~\cite{GOMS}.

In particular, when~$A$ is a field, this symbol coincides with the $n$-dimensional tame symbol from~\cite{P1}, which is a higher-dimensional generalization of the usual tame symbol. When $A$ is a certain Artinian ring over a field $k$, one obtains from the higher-dimensional Contou-Carr\`ere symbol the higher-dimensional residue. If the field $k$ is finite, then one obtains also the Witt pairing in the higher-dimensional local class field theory, see~\cite[\S\,8]{OZ1} and~\cite[\S\,9]{GOMS}.

The algebra $\LL^n(A)$ has a natural topology making it a topological $A$-module (with the discrete topology on $A$). For $n=1$, this topology is the usual topology on $A((t))$ with the base of open neighborhoods of zero given by the \mbox{$A$-mo\-dules}~$t^iA[[t]]$, ${i\in\z}$. In~\cite{GOT2}, the authors have studied continuous endomorphisms of $A$-algebra $\LL^n(A)$. In particular, it was obtained an invertibility criterion for such endomorphisms, see~\cite[Theor.\,6.8]{GOT2}, together with an explicit formula for the inverse endomorphism, see~\cite[Rem.\,6.4]{GOT2}.

\medskip

Our first main result describes how the higher-dimensional Contou-Carr\`ere symbol is changed under continuous endomorphisms of the $A$-algebra $\LL^n(A)$, see Theorem~\ref{theor:invCC}. In particular, this implies that the higher-dimensional Contou-Carr\`ere symbol is invariant under continuous automorphisms of the $A$-algebra $\LL^n(A)$, see  Corollary~\ref{cor:invCC}. We provide in Remark~\ref{rmk:general} a generalization of Theorem~\ref{theor:invCC} to continuous homomorphisms of \mbox{$A$-algebras} from $\LL^n(A)$ to $\LL^m(A)$. Also, we give an example of a non-continuous automorphism of the $A$-algebra $\LL^n(A)$ that does not preserve the higher-dimensional Contou-Carr\`ere symbol, see Proposition~\ref{prop:noninv} and Example~\ref{ex:noninv}.
\quash{This is essentially based on an example by Yekutieli of non-invariance of the higher-dimensional residue map, see~\cite[Ex.\,2.4.24]{Y}.}

The invariance result leads to a new explicit formula for the higher-dimensional Contou-Carr\`ere symbol, which is our second main result. Let us explain this in more derail.

Recall that if $A$ is a $\Q$-algebra, then the map $CC_n$ as in formula~\eqref{eq:CCintr} is given by an explicit formula, see~\cite[\S\,8.4]{GOMS} and also formulas~\eqref{eq:CC3},~\eqref{eq:CC1}, and~\eqref{eq:CC2} in Section~\ref{sec:prel} below. The explicit formula is easily applicable over $\Q$-algebras. Notice that this formula involves the standard $\log$ and $\exp$ series, whose coefficients have non-trivial denominators, and thus the formula can not be applied directly over arbitrary rings (in particular, over rings of positive characteristic).

On the other hand, when $n=1$, for any ring $A$ (not necessarily a $\Q$-algebra), there is another explicit formula for the Contou-Carr\`ere symbol, see~\cite[Intr.]{AP} and \cite[\S\,2]{OZ1}. This formula is based on the fact that any invertible Laurent series from $A((t))^*$ can be decomposed into an infinite product, where almost all factors are of type $1+a_l t^l$ with~$l\in\z$, $a_l \in A$.

When $n=2$, the situation is more delicate. The decomposition of an element from~$A((t_1))((t_2))^*$ into an infinite product, where almost all factors are of type $1 + a_{l_1,l_2} t_1^{l_1} t_2^{l_2}$ with $l_1,l_2\in\z$, $a_{l_1,l_2} \in A$, was obtained in~\cite[Prop.\,3.14]{OZ1} only when the nil-radical of the ring $A$ is a nilpotent ideal (for example, this holds if $A$ is Noetherian). Correspondingly, an explicit formula for the two-dimensional Contou-Carr\`ere symbol was obtained in~\cite{OZ1} only over such rings, see~\cite[Def.\,3.5, Prop.\,3.16]{OZ1}. Nevertheless, the definition of the two-dimensional Contou-Carr\`ere symbol was given in~\cite{OZ1} over an arbitrary ring~$A$ by means of generalized commutators in categorical central extensions, whose construction uses, in general, Nisnevich coverings of the scheme $\Spec(A)$.

For arbitrary $n\geqslant 1$, the definition of the higher-dimensional Contou-Carr\`ere symbol over any ring is given by means of boundary maps for algebraic $K$-groups, see~\cite[\S\,7]{OZ1} and~\cite[\S\,8.2]{GOMS}.

In this paper, we obtain an explicit formula for the higher-dimensional Contou-Carr\`ere symbol, which is applicable over any ring $A$ and does not use algebraic $K$-theory, see Definition~\ref{defin:CCtilde}, formula~\eqref{eq:inverse}, Theorem~\ref{theor:tilde}, and also Remark~\ref{rmk:explalt}. The formula is quite elementary and uses only products of series, the higher-dimensional residue map, and a canonical map $\pi\colon \LL^n(A)^*\to A^*$, which is induced by a natural decomposition of the group $\LL^n(A)^*$, see formula~\eqref{eq:maps} in Section~\ref{sec:prel}. The main idea for the construction of this explicit formula comes from the above invariance property of the higher-dimensional Contou-Carr\`ere symbol combined with an invertibility criterion for continuous endomorphisms of the $A$-algebra~$\LL^n(A)$ obtained in~\cite{GOT2}.

\medskip

We are grateful to Alexei Rosly for helpful remarks concerning the exposition of the paper.

\section{Preliminaries and notation}\label{sec:prel}

For short, by a ring, we mean a commutative associative unital ring and similarly for algebras.

Throughout the paper, $A$ denotes a ring and $n\geqslant 1$ is a positive integer. Let $\LL(A):=A((t)):=A[[t]][t^{-1}]$ be the ring of Laurent series over $A$ and let $\LL^n(A):=A((t_1))\ldots((t_n))=\big(\LL^{n-1}(A)\big)((t_n))$ be the ring of iterated Laurent series over $A$. Explicitly, an iterated Laurent series $f\in\LL^n(A)$ has a form $f=\sum\limits_{l\in\z^n}a_lt^l$, where $a_l\in A$, $t^l:=t_1^{l_1}\ldots t_n^{l_n}$ for $l=(l_1,\ldots,l_n)\in\z^n$, and the series satisfies a certain condition on its support, see, e.g.,~\cite[\S\,3.1]{GOMS}.

There is a natural topology on $\LL^n(A)$ such that $\LL^n(A)$ is a topological group with respect to the addition of iterated Laurent series. The base of open neighborhoods of zero in $\LL(A)$ is given by \mbox{$A$-submodules} \mbox{$U_i:=t^i A[[t]]$}, $i\in \Z$. The base of open neighborhoods of zero in $\LL^n(A)$ is given by \mbox{$A$-submodules}
$$
\mbox{$U_{i, \{V_j \}}:=\Big(\bigoplus\limits_{j<i} t_n^j\cdot V_j \Big)$} \,\oplus \, \mbox{$ t_n^i\cdot
\LL^{n-1}(A)[[t_n]]\,,$}
$$
where $i\in\Z$ and for each $j$, $j<i$, the $A$-module $V_j$ is from the base of open neighborhoods of zero in $\LL^{n-1}(A)$. See more details on the topology, e.g., in~\cite[\S\S\,3.2, 3.3]{GOMS} and~\cite[\S\,2]{GOT2}.

\medskip

Let $L^n\gm$ be a group functor on the category of rings that sends a ring $A$ to the group of invertible elements $\LL^n(A)^*$ in the ring $\LL^n(A)$. We have an embedding $\gm\hookrightarrow L^n\gm$ given by constant series. Let~$\uz$ be a group functor on the category of rings that sends a ring $A$ to the group $\uz(A)$ of all $\z$-valued locally constant functions on $\Spec(A)$ with the Zariski topology. We have an embedding
$$
\uz^n\hookrightarrow L^n\gm\,,\qquad \underline{l}\longmapsto t^{\underline{l}}\,,
$$
where $\underline{l}=(\underline{l}_1,\ldots\underline{l}_n)\in\uz^n(A)$ is a locally constant $\z^n$-valued function on $\Spec(A)$. The element $t^{\underline{l}}\in \LL^n(A)$ is defined naturally, see details in~\cite[\S\,4.2]{GOMS}. For example, suppose that $n=1$ and there is a decomposition into a product of rings $A\simeq B\times B'$ such that $\underline{l}\in\uz(A)$ takes the value $m\in\z$ on~$\Spec(B)$ and takes the value $m'\in\z$ on $\Spec(B')$. Then $t^{\underline{l}}$ is equal to the element $(t^{m},t^{m'})\in \LL(B)\times\LL(B')\simeq \LL(A)$.

Define a lexicographical order on $\z^n$ such that $(l_1,\ldots,l_n)\leqslant (l'_1,\ldots,l'_n)$ if and only if either $l_n < l'_n$, or $l_n=l'_n$ and $(l_1, \ldots, l_{n-1}) \leqslant (l'_1, \ldots, l'_{n-1})$. For short, we put ${0:=(0,\ldots,0)\in\z^n}$. Let $\vv_{n,+}(A)$ be the subgroup of $L^n\gm(A)$ that consists of all iterated Laurent series of type $1+\sum\limits_{l>0}a_l t^l$. Let $\vv_{n,-}(A)$ be the subgroup of $L^n\gm(A)$ that consists of all iterated Laurent series of type $1+\sum\limits_{l<0}a_l t^l$ such that $\sum\limits_{l<0}a_l t^l$ is a nilpotent element of $\LL^n(A)$. Then there is the following isomorphism of group functors, see~\cite{CC1},~\cite{CC2} for $n=1$ and~\cite[Prop.\,4.3]{GOMS} for arbitrary $n$:
\begin{equation}\label{eq:decv}
\lo^n\gm\simeq\uz^n\times\gm\times\vv_{n,+}\times\vv_{n,-}\,.
\end{equation}
The decomposition~\eqref{eq:decv} defines morphisms of group functors
\begin{equation}\label{eq:maps}
\nu\,:\,L^n\gm\longrightarrow \uz^n\,,\qquad \pi\,:\,L^n\gm\longrightarrow \gm\,.
\end{equation}
In particular, for any $i$, $1\leqslant i\leqslant n$, we have that $\nu(t_i)=(0,\ldots,0,1,0,\ldots,0)$, where $1$ stands on the $i$-th place. The morphism $\nu$ is the classical discrete valuation when $A$ is a field and $n=1$. Let a group functor $(L^n\gm)^0$ be the kernel of this morphism~$\nu$. Clearly, we have a decomposition of group functors
\begin{equation}\label{eq:dec1}
L^n\gm\simeq \uz^n\times (L^n\gm)^0\,.
\end{equation}

Suppose for an element $f=\sum\limits_{l\in\Z}a_lt^l$ in $\LL^n(A)^*$, the function $\nu(f)\in\uz^n(A)$ is constant on~$\Spec(A)$, that is, we have $\nu(f)=l_0\in\z^n$. Then $\pi(f)$ coincides with the coefficient~$a_{l_0}$ modulo the nilradical of the ring $A$, see~\cite[Lem.\,4.7(i)]{GOMS}. However, in general,~$\pi(f)$ is not equal to $a_{l_0}$ in $A$, see~\cite[Ex.\,4.5(iii)]{GOMS}.

\medskip

Let $(L^n\gm)^{\sharp}(A)$ consist of all iterated Laurent series $\sum\limits_{l\in\z^n}a_l t^l\in\LL^n(A)$ such that the iterated Laurent series $1-\sum\limits_{l\leqslant 0}a_lt^l$ is a nilpotent element of $\LL^n(A)$. Then $(L^n\gm)^{\sharp}$ is a group subfunctor in $L^n\gm$ and we have the following decomposition of the group functor~$(L^n\gm)^0$:
\begin{equation}\label{eq:dec2}
(\lo^n\gm)^0=\gm\cdot(\lo^n\gm)^{\sharp}\,.
\end{equation}
In addition, we have the following topological characterization of $(L^n\gm)^{\sharp}(A)$. For any $f\in\LL^n(A)$, one has $f\in (L^n\gm)^{\sharp}(A)$ if and only if the sequence $\{(f-1)^i\}$, $i\in\N$, tends to zero in $\LL^n(A)$, see~\cite[Def.\,3.7, Prop.\,3.8]{GOMS}. It follows that if $A$ is a $\Q$-algebra, then for any $f\in (L^n\gm)^{\sharp}(A)$, the series $\log(f):=\sum\limits_{i\geqslant 1}(-1)^{i+1}\frac{(f-1)^i}{i}$ converges in $\LL^n(A)$, see~\cite[\S\,3.3]{GOMS}.

\medskip

Let $\End^{\rm c,alg}_A\big(\LL^n(A)\big)$ be the monoid of all continuous endomorphisms of the \mbox{$A$-algebra}~$\LL^n(A)$. An endomorphism $\phi\in\End^{\rm c,alg}_A\big(\LL^n(A)\big)$ is determined uniquely by the collection $\phi(t_1),\ldots,\phi(t_n) \in\LL^n(A)^*$. Given a collection $\varphi_1,\ldots,\varphi_n \in\LL^n(A)^*$, we have an $(n\times n)$-matrix $\big(\nu(\varphi_1),\ldots,\nu(\varphi_n)\big)$ with entries in $\uz(A)$, where we consider elements $\nu(\varphi_i)\in\uz^n(A)$ as columns. There is an endomorphism $\phi\in\End^{\rm c,alg}_A\big(\LL^n(A)\big)$ with $\phi(t_i)=\varphi_i$, $1\leqslant i\leqslant n$, if and only if the matrix $\big(\nu(\varphi_1),\ldots,\nu(\varphi_n)\big)$ is upper-triangular and its diagonal elements are positive point-wise on $\Spec(A)$, see~\cite[Theor.\,4.7]{GOT2}.

For an endomorphism $\phi\in\End^{\rm c,alg}_A\big(\LL^n(A)\big)$, let $\Upsilon(\phi)$ be the \mbox{$(n\times n)$-matrix} $\big(\nu(\phi(t_1)),\ldots,\nu(\phi(t_n))\big)$ with entries in $\uz(A)$. Also, put $d(\phi):=\det\big(\Upsilon(\phi)\big)\in\uz(A)$. For any $f\in \LL^n(A)^*$, we have the following equality in $\uz^n(A) $, see~\cite[Prop.\,3.10]{GOT2}:
\begin{equation}\label{eq:invnu}
\nu(\phi(f))=\Upsilon(\phi)\cdot\nu(f)\,.
\end{equation}

An endomorphism $\phi$ is invertible if and only if the matrix~$\Upsilon(\phi)$ is invertible, see~\cite[Theor.\,6.8]{GOT2}. Explicitly, the last condition means that the upper-triangular matrix $\Upsilon(\phi)$ has units on the diagonal, or, equivalently, that~$d(\phi)=1$.

\medskip

Let $\Omega^1_{\LL^n(A)/A}$ be the $\LL^n(A)$-module of K\"ahler differentials of the ring $\LL^n(A)$ over $A$ and put $\Omega^n_{\LL^n(A)/A}:=\bigwedge^n_{\LL^n(A)}\Omega^1_{\LL^n(A)/A}$. Let $\widetilde{\Omega}^1_{\LL^n(A)}$ be the quotient of the $\LL^n(A)$-module~$\Omega^1_{\LL^n(A)/A}$ by the submodule generated by elements ${df-\sum\limits_{i=1}^n\frac{\partial f}{\partial t_i}}dt_i$, where ${f\in\LL^n(A)}$. Then $\widetilde{\Omega}^1_{\LL^n(A)}$ is a free $\LL^n(A)$-module of rank $n$. Put ${\widetilde{\Omega}^n_{\LL^n(A)}:=\bigwedge^n_{\LL^n(A)}\widetilde{\Omega}^1_{\LL^n(A)}}$. One has an $A$-linear residue map
$$
\res\,:\,\widetilde{\Omega}^n_{\LL^n(A)}\longrightarrow A\,,\qquad \mbox{$\sum\limits_{l\in\z^n}a_lt^l\cdot dt_1\wedge\ldots\wedge dt_n\longmapsto a_{-1\ldots-1}$}\,.
$$
We denote the natural composition
$$
\Omega^n_{\LL^n(A)/A}\longrightarrow \widetilde{\Omega}^n_{\LL^n(A)}\stackrel{\res}\longrightarrow A
$$
also by $\res$.

Any endomorphism $\phi\in\End^{\rm c,alg}_A\big(\LL^n(A)\big)$ defines naturally an $A$-linear endomorphism of~$\widetilde{\Omega}^n_{\LL^n(A)}$, which we denote also by $\phi$ for simplicity. We have the following equality, see~\cite[Cor.\,5.4]{GOT2}:
\begin{equation}\label{eq:invres}
\res\big(\phi(\omega)\big)=d(\phi)\res(\omega)\,.
\end{equation}

\medskip

In~\cite{GOMS} the authors have introduced and studied a special class of ind-affine schemes called thick ind-cones, see~\cite[\S\,5.4]{GOMS} for more details. The point is that many functors that occur in the context of iterated Laurent series are represented by ind-affine schemes which are products of ind-flat ind-affine schemes over~$\z$ and thick ind-cones (in the one-dimensional case, the corresponding ind-affine schemes are just ind-flat over $\z$). Such products possess many nice properties, mainly concerning regular functions on them. In particular, if an ind-affine scheme~$X$ is a product of an ind-flat ind-affine scheme over~$\z$ and a thick ind-cone, then the natural homomorphism between algebras of regular functions ${\OO(X)\to \OO(X_{\Q})}$ is injective, see~\cite[Prop.\,5.17]{GOMS}. Besides, for all such ind-schemes $X$ and $Y$, their product~${X\times Y}$ is also isomorphic to a product of an ind-flat ind-affine scheme over $\z$ and a thick ind-cone, see~\cite[Lem.\,5.13]{GOMS}.

The functors $L^n\gm$, $(L^n\gm)^0$, $(L^n\gm)^{\sharp}$, $\vv_{n,+}$, and $\vv_{n,-}$ are represented by ind-affine schemes which are products of ind-flat ind-affine schemes over $\z$ and thick ind-cones, see~\cite[\S\,6.3]{GOMS}.

The assignment
$$
A\longmapsto \End_A^{\rm c,alg}\big(\LL^n(A)\big)
$$
is a functor, which we denote by $\Ec nd^{\rm \,c,alg}(\LL^n)$, see~\cite[\S\,5]{GOT2}. This functor is also represented by an ind-affine scheme which is a product of an ind-flat ind-affine scheme over $\z$ and a thick ind-cone, see~\cite[Prop.\,5.1]{GOT2}.

Given a functor on the category of rings which is represented by an ind-scheme, we use the same notation for the representing ind-scheme as for the initial functor.

\medskip

By~\cite[Prop.\,8.15]{GOMS}, there is a unique multilinear (anti)symmetric map
$$
\sgn\,:\,(\z^n)^{\times(n+1)}\longrightarrow \z/2\z
$$
such that for all $l_1,\ldots,l_n\in \z^n$, we have $\sgn(l_1,l_1,l_2,\ldots,l_n)\equiv\det(l_1,l_2,\ldots,l_n)\pmod{2}$. It follows that the map $\sgn$ is invariant under automorphisms of~$\z^n$. One has (equivalent) explicit formulas for the map $\sgn$, see the proof of~\cite[Prop.\,8.15]{GOMS} and \cite[Rem.\,8.16]{GOMS}.

By~\cite[Prop.\,8.22, Rem.\,8.21]{GOMS} (see also~\cite{GO1}), for any $\Q$-algebra $A$, there is a unique multilinear antisymmetric map
$$
CC_n\,:\,\big(\LL^n(A)^*\big)^{\times (n+1)}\lrto A^*
$$
that satisfies the following properties:
\begin{itemize}
\item[(i)]
if $f_1\in (\lo^n\gm)^{\sharp}(A)$, then
\begin{equation}\label{eq:CC3}
CC_n(f_1,f_2,\ldots,f_{n+1})=\exp\,\res\left(\log(f_1)\,\frac{df_2}{f_2}\wedge\ldots\wedge
\frac{df_{n+1}}{f_{n+1}}\right)\,,
\end{equation}
where the standard series $\exp$ is applied to a nilpotent element of the ring $A$;
\item[(ii)]
if $f_1\in A^*$, then
\begin{equation}\label{eq:CC1}
CC_n(f_1,f_2\ldots,f_{n+1})=f_1^{\,\det(\nu(f_2),\,\ldots,\,\nu(f_{n+1}))}\,,
\end{equation}
\item[(iii)]
for all ${{\underline{l}_1},\ldots,{\underline{l}_{n+1}}\in\uz(A)}$, there is an equality
\begin{equation}\label{eq:CC2}
CC_n(t^{\underline{l}_1},\ldots,t^{\underline{l}_{n+1}})=(-1)^{\sgn({\underline{l}_1},\,\ldots,\,{\underline{l}_{n+1}})}\,.
\end{equation}
\end{itemize}

It is not not hard to deduce from this definition of $CC_n$ that for all elements ${f_1,\ldots,f_n\in\LL^n(A)^*}$, there is an equality
\begin{equation}  \label{eq:easy_St}
CC_n(f_1, f_1, f_2, \ldots, f_n) = (-1)^{\det(\nu(f_1),\,\ldots,\,\nu(f_n))}\,.
\end{equation}
Indeed, by~\cite[Lem.\,8.24]{GOMS}, we have that $CC_n(f_1, f_1,f_2, \ldots, f_n)= CC_n(-1, f_1, \ldots, f_n)$ and then we apply formula~\eqref{eq:CC1}.

\medskip

Clearly, the map~$CC_n$ is functorial with respect to a $\Q$-algebra~$A$, that is, we have a morphism of functors
$CC_n\colon (L^n\gm)^{\times(n+1)}_{\Q}\to (\gm)_{\Q}$ on the category of $\Q$-algebras. Since the functor~$L^n\gm$ is represented by a product of an ind-flat ind-affine scheme over $\z$ and a thick ind-cone, the above discussion on such ind-schemes implies that there is at most one extension of the above map $CC_n$ to a morphism of functors $(L^n\gm)^{\times (n+1)}\to\gm$ on the category of all rings. In addition, this extension must be automatically multilinear, antisymmetric and must satisfy formula~\eqref{eq:easy_St}, which also follows from the theory of thick ind-cones.

Actually, such extension does exist. Namely, based on the works of Grayson~\cite{Gr} and Kato~\cite{K1}, one constructs a boundary map between algebraic $K$-groups
$$
\partial_{m+1}\,:\,K_{m+1}\big(\LL(A)\big)\longrightarrow K_m(A)\,,\qquad m\geqslant 0\,,
$$
which is functorial with respect to a ring $A$, see~\cite[\S\,7]{OZ1} and~\cite[\S\,7.3]{GOMS}. Recall that for any ring~$B$, there is a surjective homomorphism ${\det\colon K_1(B)\to B^*}$, which has a section given by a canonical embedding $B^*\subset K_1(B)$. This allows to define the composition
\begin{equation}\label{eq:K}
\begin{CD}
\big(\LL^n(A)^*\big)^{\times (n+1)}@>>> K_1\big(\LL^n(A)\big)^{\times (n+1)}@>>>\\@>>> K_{n+1}\big(\LL^n(A)\big)@>{\partial_2  \cdot \ldots \cdot \partial_{n+1}}>>  K_1(A)
@>\det>> A^*\,,
\end{CD}
\end{equation}
where the second map is the product between algebraic $K$-groups.

A non-trivial theorem is that if $A$ is a $\Q$-algebra, then the composition in formula~\eqref{eq:K} coincides with the map defined above explicitly by formulas~\eqref{eq:CC3},~\eqref{eq:CC1}, and~\eqref{eq:CC2}, see~\cite[Theor.\,8.17]{GOMS}. Thus we denote also by $CC_n$ the morphism of functors ${(L^n\gm)^{\times (n+1)}\to\gm}$ obtained from formula~\eqref{eq:K} and call it a higher-dimensional Contou-Carr\`ere symbol.

\section{Invariance of the higher-dimensional Contou-Carr\`ere symbol}

\begin{theor}\label{theor:invCC}
Let $\phi\colon \LL^n(A)\to\LL^n(A)$ be a continuous endomorphism of the \mbox{$A$-algebra}~$\LL^n(A)$. Then for all elements $f_1,\ldots,f_{n+1}\in\LL^n(A)^*$, there is an equality in $A^*$
\begin{equation}\label{eq:invCC}
CC_n\big(\phi(f_1),\ldots,\phi(f_{n+1})\big)=CC_n(f_1,\ldots,f_{n+1})^{d(\phi)}\,.
\end{equation}
\end{theor}
\begin{proof}
Both sides of formula~\eqref{eq:invCC} are regular functions on the ind-affine scheme
${X:=\Ec nd^{\rm \,c, alg}\big(\LL^n)\times (L^n\gm)^{\times(n+1)}}$. The scheme $X$ is isomorphic to a product of an ind-flat ind-affine scheme over $\z$ and a thick ind-cone. Therefore, by results mentioned in Section~\ref{sec:prel}, the natural homomorphism between algebras of regular functions $\OO(X)\to\OO(X_{\Q})$ is injective. Hence it is enough to prove formula~\eqref{eq:invCC} when $A$ is a $\Q$-algebra, which we assume from now on.

Further, both sides of formula~\eqref{eq:invCC} are multilinear and antisymmetric with respect to~$f_1,\ldots,f_{n+1}$. Thus decompositions~\eqref{eq:dec1} and~\eqref{eq:dec2} imply that it is enough to prove formula~\eqref{eq:invCC} in the following three cases:
\begin{itemize}
\item[(i)]
we have that $f_1\in (L^n\gm)^{\sharp}(A)$;
\item[(ii)]
we have that $f_1\in A^*$;
\item[(iii)]
there is a collection $\underline{l}_1,\ldots,\underline{l}_{n+1}\in\uz^n(A)$ such that $f_i=t^{\underline{l}_i}$, where $1\leqslant i\leqslant n+1$.
\end{itemize}

Suppose that condition~(i) holds. Since $\phi$ is continuous, the topological characterization of the group $(L^n\gm)^{\sharp}(A)$ given in Section~\ref{sec:prel} implies that $\phi$ preserves $(L^n\gm)^{\sharp}(A)$, whence $\phi(f_1)\in (L^n\gm)^{\sharp}(A)$. Consequently, by formulas~\eqref{eq:CC3} and~\eqref{eq:invres}, the left hand side of~\eqref{eq:invCC} is equal to
$$
\exp\,\res\left(\log(\phi(f_1))\,\frac{d\phi(f_2)}{\phi(f_2)}\wedge\ldots\wedge
\frac{d\phi(f_{n+1})}{\phi(f_{n+1})}\right)=
\exp\,\res\left(\phi\left(\log(f_1)\,\frac{df_2}{f_2}\wedge\ldots\wedge
\frac{df_{n+1}}{f_{n+1}}\right)\right)=
$$
$$
=\exp\,\left(d(\phi)\res\left(\log(f_1)\,\frac{df_2}{f_2}\wedge\ldots\wedge
\frac{df_{n+1}}{f_{n+1}}\right)\right)=\exp\,\left(\res\left(\log(f_1)\,\frac{df_2}{f_2}\wedge\ldots\wedge
\frac{df_{n+1}}{f_{n+1}}\right)\right)^{d(\phi)}\,.
$$
Again by formula~\eqref{eq:CC3}, this equals to the right hand side of~\eqref{eq:invCC}.

Now suppose that condition~(ii) holds. Since $\phi$ is $A$-linear, we have that ${\phi(f_1)=f_1}$. By~\cite[Prop.\,8.4]{GOMS}, there are equalities
$$
\det\big(\nu(f_2),\,\ldots,\,\nu(f_{n+1})\big)=\res \left(\frac{df_2}{f_2}  \wedge \ldots \wedge \frac{df_{n+1}}{f_{n+1}} \right)\,,
$$
$$
\det\big(\nu(\phi(f_2)),\,\ldots,\,\nu(\phi(f_{n+1}))\big)=\res \left(\frac{d\phi(f_2)}{\phi(f_2)}  \wedge \ldots \wedge \frac{d\phi(f_{n+1})}{\phi(f_{n+1})} \right)\,.
$$
By formulas~\eqref{eq:invres} and~\eqref{eq:CC1}, this proves~\eqref{eq:invCC} when $f_1\in A^*$.

Finally, suppose that condition~(iii) holds. There is a decomposition into a finite product of rings
\begin{equation}\label{eq:decompring}
\mbox{$A\simeq\prod\limits_{j=1}^N A_j$}
\end{equation}
such that for all $i$ and $j$ with $1\leqslant i\leqslant n+1$, $1\leqslant j\leqslant N$, the restriction of the function $\underline{l}_i$ to~$\Spec(A_j)$ is constant. Since both sides of formula~\eqref{eq:invCC} are functorial with respect to the terms of decomposition~\eqref{eq:decompring} (actually, with respect to any $A$-algebra), it is enough to prove formula~\eqref{eq:invCC} for each ring $A_j$ separately. Thus we may assume that all $\underline{l}_i$ are constant, that is, are elements of $\z^n\subset\uz^n(A)$.

It follows from multilinearity and the antisymmetric property that it is enough to consider the case when the collection $(f_1,\ldots,f_{n+1})$ is of type $(t_{p_1},t_{p_1},t_{p_2},\ldots,t_{p_n})$ for some $1\leqslant p_1,\ldots,p_n\leqslant n$. By~\cite[Lem.\,8.24]{GOMS}, the values of both sides of formula~\eqref{eq:invCC} at such collection coincide with their values at the collection $(-1,t_{p_1},t_{p_2},\ldots,t_{p_n})$. Thus we are reduced to the case~(ii) considered above.

\end{proof}

\begin{corol}\label{cor:invCC}
If $\phi$ is a continuous automorphism of the $A$-algebra $\LL^n(A)$, then for all elements $f_1,\ldots,f_{n+1}\in\LL^n(A)^*$, there is an equality in $A^*$
$$
CC_n\big(\phi(f_1),\ldots,\phi(f_{n+1})\big)=CC_n(f_1,\ldots,f_{n+1})\,.
$$
\end{corol}
\begin{proof}
Since the endomorphism $\phi$ is invertible, we have that $d(\phi)=1$, see Section~\ref{sec:prel}. Thus the corollary follows from Theorem~\ref{theor:invCC}.
\end{proof}

\begin{rmk} \label{rmk:general}
Theorem~\ref{theor:invCC} has the following generalization. Let $\phi\colon\LL^n(A)\to\LL^m(A)$ be a continuous homomorphism of $A$-algebras. In particular, by~\cite[Cor.\,4.8]{GOT2}, we have that~${n\leqslant m}$. One defines invariants $\sgn(\phi)$ and $d(\phi)$ in $\uz(A)$, see~\cite[\S\,5]{GOT2} (when $m=n$, we have that $\sgn(\phi)=1$ and $d(\phi)$ is $\det\big(\Upsilon(\phi)\big)$ as above). The homomorphism $\phi$ defines also a collection $q_1,\ldots,q_{m-n}$ of elements from $\uz(A)$ such that for any form $\omega\in\widetilde{\Omega}^n_{\LL^n(A)}$, one has the following equality, see~\cite[Prop.\,5.3]{GOT2}:
$$
\res\left(\phi(\omega)\wedge\frac{dt_{q_1}}{t_{q_1}}\wedge\ldots\wedge\frac{dt_{q_{m-n}}}{t_{q_{m-n}}}\right)=\sgn(\phi)d(\phi)\res(\omega)\,.
$$
A similar argument as in the proof of Theorem~\ref{theor:invCC} shows that for all elements $f_1,\ldots,f_{n+1}\in\LL^n(A)^*$, there is an equality in $A^*$
$$
CC_m\big(\phi(f_1),\ldots,\phi(f_{n+1}),t_{q_1},\ldots,t_{q_{m-n}}\big)=CC_n(f_1,\ldots,f_{n+1})^{\sgn(\phi)d(\phi)}\,.
$$
\end{rmk}

\quash{
\begin{rmk}
It is not clear how to deduce Corollary~\ref{cor:invCC} directly from the $K$-theoretic definition of the Contou-Carr\`ere symbol, without using the explicit formula for it. However, it is not hard to deduce the following weaker statement. Given a ring $B$, let $\psi\colon \LL(B)\to\LL(B)$ be a continuous automorphism (not necessarily $B$-linear) such that $\psi(B)\subset B$ and ${\psi(t)\in B[[t]]\subset B((t))=\LL(B)}$. Then, using the explicit definition of the boundary map between algebraic $K$-groups, see~\cite[\S\,7]{OZ1} or~\cite[\S\,7.3]{GOMS}, one shows easily that the map $\partial_m\colon K_{m+1}\big(\LL(B)\big)\to K_m(B)$ commutes with the automorphism $\psi$. By induction, this implies that the higher-dimensional Contou-Carr\`ere symbol is invariant under continuous automorphism of the $A$-algebra $\LL^n(A)$ for all automorphisms $\phi$ such that for any $i$, $1\leqslant i\leqslant n$, we have
$$
\phi(t_i)\in A((t_1))\ldots((t_{i-1}))[[t_i]]\,.
$$
For an arbitrary ring $A$ with nilpotents, these are very special automorphisms, which follows from the explicit description of all continuous automorphisms in terms of the values $\phi(t_i)$ given in Section~\ref{sec:prel}.
\end{rmk}
}

\medskip

Now let us show that if an automorphism of the $A$-algebra $\LL^n(A)$ is not continuous, then, in general, it does not preserve the higher-dimensional Contou-Carr\`ere symbol.

Let $\phi$ be an automorphism of the $A$-algebra $\LL^n(A)$ (not necessarily continuous). Then~$\phi$ defines naturally an $A$-linear automorphism of~$\Omega^n_{\LL^n(A)/A}$, which we denote also by $\phi$ for simplicity. Let $\varepsilon$ be a formal variable that satisfies $\varepsilon^2=0$. By $\phi_{\varepsilon}$ denote the induced automorphism of the $A[\varepsilon]$-algebra $\LL^n\big(A[\varepsilon]\big)\simeq \LL^n(A)[\varepsilon]$.

\begin{prop}\label{prop:noninv}
Assume that there is a differential form $\omega\in\Omega^n_{\LL^n(A)/A}$ such that
$$
\res\big(\phi(\omega)\big)\ne \res(\omega)\,.
$$
Then there is a collection ${f_1,\ldots,f_{n+1}\in \LL^n\big(A[\varepsilon]\big)^*}$ such that
$$
CC_n\big(\phi_{\varepsilon}(f_1),\ldots,\phi_{\varepsilon}(f_{n+1})\big)\ne CC_n(f_1,\ldots,f_{n+1})\,.
$$
\end{prop}
\begin{proof}
Suppose that $\phi_{\varepsilon}$ preserves the higher-dimensional Contou-Carr\`ere symbol. By~\cite[Prop.\,8.19]{GOMS}, for any collection $f_1,\ldots,f_n\in\LL^n(A)^*$ and any element $g\in\LL^n(A)$, there is an equality
\begin{equation}\label{eq:CCres}
CC_n (1+g\,\varepsilon,f_1,\ldots,f_n)=1+\mathop{\rm res}\left(g\,\frac{df_1}{f_1}\wedge\ldots\wedge\frac{df_n}{f_n}\right)\varepsilon\,,
\end{equation}
where $CC_n$ is applied to a collection of invertible elements of the ring $\LL^n\big(A[\varepsilon]\big)$. Thus we see that $\phi$ preserves the residue in the right hand side of formula~\eqref{eq:CCres}.

Note that $\LL^n(A)$ is additively generated by invertible elements, see~\cite[Ex.\,2.3(iii)]{GOMilnor}. It follows that the additive group $\Omega^n_{\LL^n(A)/A}$ is generated by differential forms as in the right hand side of formula~\eqref{eq:CCres}. Hence $\phi$ preserves the residue of all differential forms in~$\Omega^n_{\LL^n(A)/A}$, which contradicts the assumption of the proposition.
\end{proof}

\begin{examp} \label{ex:noninv}
Yekutieli has given in~\cite[Ex.\,2.4.24]{Y} the following construction of a non-continuous automorphism~$\phi$ of the field~$k((t_1))((t_2))$ over a field $k$ of zero characteristic such that $\phi$ does not preserve the residue. Let $f\in k((t_1))$ be an element which is transcendental over the subfield $k(t_1)\subset k((t_1))$. It follows from Hensel's lemma that there is a (non-unique) section $\phi_0\colon k((t_1))\to k((t_1))[[t_2]]$ of the natural homomorphism ${k((t_1))[[t_2]]\to k((t_1))}$ such that $\phi_0$ is identical on $k(t_1)$ and $\phi_0(f)=f+t_2$. One shows directly that $\phi_0$ extends to an automorphism of the ring $k((t_1))[[t_2]]$ that sends $t_2$ to itself. Thus we obtain an automorphism $\phi$ of the field $k((t_1))((t_2))$ over $k$ such that
$$
\phi(t_1)=t_1\,,\qquad \phi(t_2)=t_2\,,\qquad \phi(f)=f+t_2\,.
$$
It follows that
$$
\phi\left(\frac{dt_1}{t_1}\wedge\frac{df}{t_2}\right)=\frac{dt_1}{t_1}\wedge\frac{df}{t_2}+\frac{dt_1}{t_1}\wedge\frac{dt_2}{t_2}\,.
$$
Since $f\in k((t_1))$, the image of $\frac{dt_1}{t_1}\wedge\frac{df}{t_2}$ under the natural map ${\Omega^2_{k((t_1))((t_2))/k}\to \widetilde{\Omega}^2_{k((t_1))((t_2))}}$ equals zero. Hence we have
$$
\res\left(\frac{dt_1}{t_1}\wedge\frac{df}{t_2}\right)=0\,,\qquad \res\,\phi\left(\frac{dt_1}{t_1}\wedge\frac{df}{t_2}\right)=1\,.
$$
Therefore by our Proposition~\ref{prop:noninv}, the automorphism $\phi_{\varepsilon}$ does not preserve the two-dimensional Contou-Carr\`ere symbol over the ring $k[\varepsilon]((t_1))((t_2))$. Explicitly, we have that
$$
CC_2\left(1+\phi\left(\frac{f}{t_2}\right)\varepsilon,\phi(t_1),\phi(f)\right)\ne CC_2\left(1+\frac{f}{t_2}\,\varepsilon,t_1,f\right)\,,
$$
which follows from formula~\eqref{eq:CCres}.
\end{examp}

\section{A new formula for the higher-dimensional Contou-Carr\`ere symbol}

We will use the following simple fact. Let $L$ be a free $\z$-module of rank $2n$ with a basis $\{e_1,\ldots,e_n,x_1,\ldots,x_n\}$. Let $\Kc$ be the set of all pairs $(K,\kappa)$, where $K$ is a (possibly, empty) subset of $\{1,\ldots,n\}$ and ${\kappa\colon K\hookrightarrow \{1,\ldots,n\}}$ is an order-preserving injective map (which might differ from the initial embedding). For each $(K,\kappa)\in\Kc$, define an element
\begin{equation}\label{eq:Rosly}
\mbox{$v_{(K,\kappa)}:=v_1\wedge\ldots\wedge v_n\in\Lambda^n L$}\,,
\end{equation}
where for each $i$, $1\leqslant i\leqslant n$, we put
$$
v_i:=\left\{
\aligned
&e_i,\quad \qquad \quad {\rm if}\quad i\notin K\,,\\
&e_i+x_{\kappa(i)},\quad {\rm if}\quad i\in K\,.\\
\endaligned
\right.
$$
For example, we have that
\begin{equation}\label{eq:exampv}
v_{\varnothing}=e_1\wedge\ldots\wedge e_n\,,\qquad v_{(\{1,\ldots, n\},{\rm id})}=(e_1+x_1)\wedge\ldots\wedge(e_n+x_n)\,.
\end{equation}

\begin{lemma}\label{lem:basis}
The set $\{v_{(K,\kappa)}\}$, where $(K,\kappa)\in\Kc$, is a basis of the $\z$-module $\Lambda^nL$.
\end{lemma}
\begin{proof}
We claim that the set $\Kc$ is naturally bijective with the set of $n$-combinations from a set of $2n$ elements. Indeed, a pair $(K,\kappa)$ corresponds to the subset
$$
\mbox{$K\sqcup \overline{\kappa(K)}\subset \{1,\ldots,n\}\sqcup\{1,\ldots,n\}\simeq\{1,\ldots,2n\}$}\,,
$$
where $\overline{\kappa(K)}:=\{1,\ldots,n\}\smallsetminus \kappa(K)$. Hence the set $\Kc$ has ${2n}\choose{n}$ elements, which is also equal to the rank of $\Lambda^n L$ over $\z$.

We see that it is enough to show that the set $\{v_{(K,\kappa)}\}$, where $(K,\kappa)\in\Kc$, generates the $\z$-module $\Lambda^nL$. Consider an element ${e_{p_1}\wedge\ldots\wedge e_{p_m}\wedge x_{r_1}\wedge\ldots\wedge x_{r_{n-m}}\in\Lambda^n L}$, where ${0\leqslant m\leqslant n}$, ${1\leqslant p_1<\ldots< p_m\leqslant n}$, and ${1\leqslant r_1<\ldots< r_{n-m}\leqslant n}$ (in particular, when $m=0$, there is no $p_i$ and when $m=n$, there is no $r_j$). Let ${1\leqslant q_1<\ldots <q_{n-m}\leqslant n}$ be the complementary collection to $\{p_1,\ldots,p_m\}$ in $\{1,\ldots,n\}$. Let $\sigma$ be the permutation that sends $(1,\ldots,n)$ to $(p_1,\ldots,p_m,q_1,\ldots,q_{n-m})$. Define an element $(M,\mu)\in\Kc$, where $M:=\{q_1,\ldots,q_{n-m}\}$ and $\mu(q_i):=r_{i}$, where ${1\leqslant i\leqslant n-m}$.

Then there is an equality in $\Lambda^n L$
\begin{equation}\label{eq:wedge}
e_{p_1}\wedge\ldots\wedge e_{p_m}\wedge x_{r_1}\wedge\ldots\wedge x_{r_{n-m}}=
(-1)^{\sgn(\sigma)+n-m}\sum_{K\subset M}(-1)^{|K|}v_{(K,\kappa)}\,,
\end{equation}
where $K$ runs over all subsets of $M$ (including the empty set), $\kappa$ is the restriction of $\mu$ to $K$, and $|K|$ denotes the number of elements in the set $K$. This is obtained directly as a result of opening brackets after we replace $x_{r_i}$ by the equal expression $(e_{q_i}+x_{r_i})-e_{q_i}$ for all $i$, $1\leqslant i\leqslant n-m$.

Since the elements $e_{p_1}\wedge\ldots\wedge e_{p_m}\wedge x_{r_1}\wedge\ldots\wedge x_{r_{n-m}}$ generate $\Lambda^n L$, this finishes the proof.
\end{proof}

\medskip

Let $g=(g_1,\ldots,g_n)$ be a collection of invertible iterated Laurent series from $\LL^n(A)^*$ such that the $(n\times n)$-matrix $\big(\nu(g)\big):=\big(\nu(g_1),\ldots,\nu(g_n)\big)$ is upper-triangular with units on the diagonal. As it was explained in Section~\ref{sec:prel}, there is a unique continuous automorphism of the $A$-algebra ${\phi_g\colon \LL^n(A)\to \LL^n(A)}$ such that $\phi_g(t_i)=g_i$ for all $i$, ${1\leqslant i\leqslant n}$. For any element ${f\in \LL^n(A)^*}$, put
$$
\langle f,g\rangle=\langle f,g_1,\ldots,g_n\rangle:=\pi\big(\phi_g^{-1}(f)\big)\,,
$$
where $\pi$ is defined in Section~\ref{sec:prel}, see formula~\eqref{eq:maps}. In particular, we have the equality ${\langle f,t_1,\ldots,t_n\rangle=\pi(f)}$.

Recall that in~\cite[Rem.\,6.4, Theor.\,6.8]{GOT2}, it is given an explicit formula for the inverse to a continuous automorphism of the $A$-algebra $\LL^n(A)$.
This implies the equality
\begin{equation}\label{eq:inverse}
\langle f,g\rangle=\pi\left(\,\mbox{$\sum\limits_{l\in\z^n}\res\big(fg^{-l-1}J(g)dt_1\wedge\ldots \wedge dt_n\big)t^l$}\right)\,,
\end{equation}
where $g^{-l-1}:=g_1^{-l_1-1}\ldots g_n^{-l_n-1}$ and $J(g)\in\LL(A)$ is the Jacobian of $\phi_g$, that is, the determinant of the matrix $\left(\frac{\partial g_i}{\partial t_j}\right)$.

\medskip

Let $e:=(e_1,\ldots,e_n)$ be a row of elements in the $\z$-module $L$. For any element $l=(l_1,\ldots,l_n)\in\z^n$ considered as a column, it is defined an element $e\cdot l=\sum\limits_{i=1}^nl_ie_i\in L$.

Clearly, for any ring $A$, one generalizes the above notation and facts (in particular, Lemma~\ref{lem:basis}) by means of   replacement of  $L$ by a free $\uz(A)$-module $L$ of rank $2n$ with a basis $\{e_1,\ldots,e_n,x_1,\ldots,x_n\}$, where we denoted this module also by $L$ for short.

\begin{defin}\label{defin:CCtilde}
For any ring $A$ and any collection $f_1,\ldots,f_{n+1}\in\LL^n(A)^*$, define an element of $A^*$
\begin{equation}\label{eq:CCtilde}
\widetilde{CC}_n(f_1\ldots,f_{n+1}):=(-1)^{\sgn(\nu(f_1),\,\ldots,\,\nu(f_{n+1}))}
\prod_{(K,\kappa)\in\Kc}\langle f_1,g_{(K,\kappa)}\rangle^{C(K,\kappa)}\,,
\end{equation}
where $C(K,\kappa)$ are elements of $\uz(A)$ such that there is an equality in the free \mbox{$\uz(A)$-mo\-dule}~$\Lambda^n L$
$$
\big(e\cdot\nu(f_2)+x_1\big)\wedge\ldots\wedge\big(e\cdot\nu(f_{n+1})+x_n\big)=\sum_{(K,\kappa)\in\Kc}C(K,\kappa)v_{(K,\kappa)}
$$
with $v_{(K,\kappa)}$ being defined as in~\eqref{eq:Rosly}, and for each $(K,\kappa)\in\Kc$, we put
$$
g_{(K,\kappa)}:=(g_{1},\ldots,g_n)
$$
with
$$
g_i:=\left\{
\aligned
&t_i,\qquad \qquad \qquad \qquad \quad\,\, {\rm if}\quad i\notin K\,,\\
&t_i\cdot f_{\kappa(i)+1}\cdot t^{-\nu(f_{\kappa(i)+1})},\quad {\rm if}\quad i\in K\\
\endaligned
\right.
$$
for each $i$, $1\leqslant i\leqslant n$.
\end{defin}

Note that the elements $C(K,\kappa) \in \uz(A)$ are well-defined by Lemma~\ref{lem:basis}. Also, one checks easily that for any $(K,\kappa)\in\Kc$, the $(n\times n)$-matrix $\big(\nu(g_{(K,\kappa)})\big)$ is the identity. Hence $\langle f,g_{(K,\kappa)}\rangle$ is a well-defined element of $A^*$.

For example, we have that (cf. formula~\eqref{eq:exampv})
\begin{equation}\label{eq:exampg}
g_{\varnothing}=(t_1,\ldots,t_n)\,,\qquad g_{(\{1,\ldots, n\},{\rm id})}=(t_1\cdot f_2\cdot t^{-\nu(f_2)},\ldots,t_n\cdot f_{n+1}\cdot t^{-\nu(f_{n+1})})\,.
\end{equation}

Clearly, the map $\widetilde{CC}_n$ is functorial with respect to a ring~$A$.

\medskip

\begin{theor}\label{theor:tilde}
For any ring $A$ and any collection $f_1,\ldots,f_{n+1}\in\LL^n(A)^*$, there is an equality
$$
\widetilde{CC}_n(f_1,\ldots,f_{n+1})=CC_n(f_1,\ldots,f_{n+1})\,,
$$
where the map $CC_n$ is defined by formula~\eqref{eq:K}.
\end{theor}
\begin{proof}
As mentioned in Section~\ref{sec:prel}, the morphism of functors $(L^n\gm)^{\times(n+1)}\to \gm$ defined by formula~\eqref{eq:K} coincides after restriction to $\Q$-algebras with the morphism of functors $(L^n\gm)^{\times(n+1)}_{\Q}\to (\gm)_{\Q}$ defined by formulas~\eqref{eq:CC3},~\eqref{eq:CC1}, and~\eqref{eq:CC2}. Moreover, there is only one morphism of functors $(L^n\gm)^{\times(n+1)}\to \gm$ with this property, that is, which coincides with the given one after restriction to $\Q$-algebras. Indeed, the theory of thick ind-cones implies that the natural homomorphism ${\OO\big((L^n\gm)^{\times(n+1)}\big)\to\OO\big((L^n\gm)^{\times(n+1)}_{\Q}\big)}$ is injective. Thus we may assume that $A$ is a $\Q$-algebra, which we do from now on.

We will use the following auxiliary multilinear antisymmetric map: for any collection $f_1,\ldots,f_{n+1}\in\LL^n(A)^*$, put
$$
F(f_1,\ldots,f_{n+1}):=(-1)^{\sgn(\nu(f_1),\,\ldots,\,\nu(f_{n+1}))}\,CC_n(f_1,\ldots,f_{n+1})\in A^*\,.
$$
We claim that the map $F\colon \big(\LL^n(A)^*\big)^{\times (n+1)}\to A^*$ is also alternating, that is, for any collection $f_1,\ldots,f_{n}\in\LL^n(A)^*$, there is an equality
${F(f_1,f_1,f_2,\ldots,f_n)=1}$. Indeed, by formula~\eqref{eq:easy_St}, there is an equality
$$
F(f_1,f_1,f_2,\ldots,f_n)=(-1)^{\sgn(\nu(f_1),\nu(f_1),\,\ldots,\,\nu(f_{n}))}\, (-1)^{\det(\nu(f_1),\,\ldots,\,\nu(f_n))}\,.
$$
Further, by the properties of the map $\sgn$ from Section~\ref{sec:prel}, the right hand side of the latter formula is equal~to
$$
(-1)^{\det(\nu(f_1),\,\ldots,\,\nu(f_{n}))}\,(-1)^{\det(\nu(f_1),\,\ldots,\,\nu(f_{n}))}=1\,.
$$

For any $f \in \LL^n(A)^*$ and $\underline{l} \in \uz(A)$, the assignment $f \mapsto f^{\underline{l}} $
gives the structure of a \mbox{$\uz(A)$-mo\-dule} on the group $\LL^n(A)^*$.
Now consider a collection $f_1,\ldots,f_{n+1}\in\LL^n(A)^*$. Define a homomorphism of $\uz(A)$-modules
\begin{equation}\label{eq:alpha}
\alpha\,:\,L\longrightarrow \LL^n(A)^*\,,\qquad e_i\longmapsto t_i\,,\quad x_i\longmapsto f_{i+1}\cdot t^{-\nu(f_{i+1})}\,,
\end{equation}
where $1\leqslant i\leqslant n$. Since the map $F$ is multilinear and alternating, this defines a homomorphism of $\uz(A)$-modules
$$
\beta\,:\,\Lambda^n L\longrightarrow A^*\,,\qquad u_1\wedge\ldots\wedge u_n\longmapsto F\big(f_1,\alpha(u_1),\ldots,\alpha(u_n)\big)\,,
$$
where $u_1,\ldots,u_n\in L$. By construction, we have the equalities in $A^*$
$$
\beta\big((e\cdot\nu(f_2)+x_1)\wedge\ldots\wedge(e\cdot\nu(f_{n+1})+x_n)\big)=F(f_1,f_2,\ldots,f_{n+1})\,,
$$
$$
\beta(v_{(K,\kappa)})=F(f_1,g_{(K,\kappa)})
$$
for any $(K,\kappa)\in\Kc$. This implies that there is an equality
$$
F(f_1,\ldots,f_{n+1})=\prod_{(K,\kappa)\in\Kc} F(f_1,g_{(K,\kappa)})^{C(K,\kappa)}\,.
$$
As mentioned in Section~\ref{sec:prel}, the map $\sgn$ is invariant under automorphisms of the \mbox{$\uz(A)$-mo\-dule}~$\uz^n(A)$. Together with formula~\eqref{eq:invnu} and Corollary~\ref{cor:invCC} this implies that~$F$ is invariant under continuous automorphisms of the $A$-algebra $\LL^n(A)$. Thus for any $(K,\kappa)\in\Kc$, there is an equality
$$
F(f_1,g_{(K,\kappa)})=F(\phi_{g_{(K,\kappa)}}^{-1}(f_1),t_1,\ldots,t_n)\,.
$$
Therefore it remains to prove that for any element $f\in\LL^n(A)^*$, there is an equality
\begin{equation}\label{eq:pi}
F(f,t_1,\ldots,t_n)=\pi(f)\,.
\end{equation}
If $f\in A^*$, then by formula~\eqref{eq:CC1} both sides of the formula~\eqref{eq:pi} are equal to $f$ itself.
If $f=t^l$ for some $l\in\uz^n(A)$, or $f\in\vv_{n,+}$, or $f\in\vv_{n,-}$, then by formula~\eqref{eq:CC2} or~\eqref{eq:CC3}, respectively, both sides of formula~\eqref{eq:pi} are equal to $1$ (recall we are assuming that~$A$ is a \mbox{$\Q$-algebra}). Now formula~\eqref{eq:pi} follows from decomposition~\eqref{eq:decv}. This finishes that proof of the theorem.
\end{proof}

\begin{rmk}\label{rmk:tilde}
For some particular choices of elements $f_1,\ldots,f_{n+1}\in \LL^n(A)^*$, there are other formulas of the same type as formula~\eqref{eq:CCtilde} which also give the higher-dimensional Contou-Carr\`ere symbol (cf. Example~\ref{ex:tilde}(i),(iii) below). Indeed, it follows from the proof of Theorem~\ref{theor:tilde} that one has the equality
$$
CC_n(f_1,\ldots,f_{n+1})=(-1)^{\sgn(\nu(f_1),\,\ldots,\,\nu(f_{n+1}))}
\prod_{(K,\kappa)\in\Kc}\langle f_1,g_{(K,\kappa)}\rangle^{D(K,\kappa)}\,,
$$
where
$$
u_1\wedge\ldots\wedge u_n=\sum\limits_{(K,\kappa)\in\Kc}D(K,\kappa)v_{(K,\kappa)}\in \Lambda^n L
$$
and $u_1,\ldots,u_n\in L$ are any elements that satisfy the condition ${\alpha(u_i)=\alpha\big(e\cdot\nu(f_{i+1})+x_i\big)}$ for all $i$, $1\leqslant i\leqslant n$, with the map ${\alpha\colon L\to \LL^n(A)^*}$ defined by formula~\eqref{eq:alpha}.
\end{rmk}

\medskip

\begin{examp}\label{ex:tilde}
Here are corollaries of Theorem~\ref{theor:tilde} combined with formula~\eqref{eq:CCtilde} and Remark~\ref{rmk:tilde}.
\begin{itemize}
\item[(i)]
If $n=1$, then for all elements ${f_1,f_2\in\LL(A)^*=A((t))^*}$, there is an equality
\begin{equation}\label{eq:CConedim}
CC_1(f_1,f_2)=(-1)^{\nu(f_1)\nu(f_2)}\langle f_1,t\rangle^{\nu(f_2)-1}\langle f_1,t^{1-\nu(f_2)}\cdot f_2\rangle=
\end{equation}
$$
=(-1)^{\nu(f_1)\nu(f_2)}\pi(f_1)^{\nu(f_2)-1}\langle f_1,t^{1-\nu(f_2)}\cdot f_2\rangle\,.
$$
Indeed, in this case, $\sgn\big(\nu(f_1),\nu(f_2)\big)\equiv \nu(f_1)\nu(f_2)\pmod{2}$, the set $\Kc$ consists of two elements: the empty set $\varnothing$ and $(\{1\},{\rm id})$, and by formula~\eqref{eq:exampv}, we have
$$
e\cdot \nu(f_2)+x=\big(\nu(f_2)-1\big)v_{\varnothing}+v_{(\{1\},{\rm id})}\,.
$$
Now we apply formula~\eqref{eq:exampg}.

Note that if $f_2=t^{\underline{l}}$ for some $\underline{l} \in \uz(A)$, then the right hand side of formula~\eqref{eq:CConedim} takes the form
$$
(-1)^{\nu(f_1)\nu(f_2)}\pi(f_1)^{\nu(f_2)-1}\langle f_1,t\rangle=(-1)^{\nu(f_1)\nu(f_2)}\pi(f_1)^{\nu(f_2)}\,.
$$
This agrees with Remark~\ref{rmk:tilde}, because in this case, there is an equality ${\alpha\big(e\cdot \nu(f_2)+x\big)=\alpha\big(e\cdot \nu(f_2)\big)}$ and we have $e\cdot \nu(f_2)=\nu(f_2)v_{\varnothing}$.
\item[(ii)]
If the $(n\times n)$-matrix $\big(\nu(f_2),\ldots,\nu(f_{n+1})\big)$ is the identity, then there is an equality
$$
CC_n(f_1,f_2,\ldots,f_{n+1})=(-1)^{\sum\limits_{i=1}^n \underline{l}_i}\,\langle f_1,f_2,\ldots,f_{n+1}\rangle\,,
$$
where $\nu(f_1)=(\underline{l}_1,\ldots, \underline{l}_n)$, $\underline{l}_i\in \uz(A)$. Indeed, in this case, we have ${\sgn\big(\nu(f_1),\ldots,\nu(f_{n+1})\big)\equiv\sum\limits_{i=1}^n \underline{l}_i}\pmod{2}$ (it is enough to check this when~$\nu(f_1)$ is an element of the standard basis in  $\uz(A)$-module $\uz^n(A)$, in which case this is obvious). Further, by formula~\eqref{eq:exampv}, we have
$$
\big(e\cdot\nu(f_2)+x_1\big)\wedge\ldots\wedge\big(e\cdot\nu(f_{n+1})+x_n\big)=v_{(\{1,\ldots, n\},{\rm id})}
$$
and we apply formula~\eqref{eq:exampg}.
\item[(iii)]
Suppose that there is a collection ${1\leqslant p_1<\ldots<p_m\leqslant n}$ with ${0\leqslant m\leqslant n}$ such that ${f_2=t_{p_1},\ldots,f_{m+1}=t_{p_m}}$ and ${\nu(f_{m+2})=\ldots=\nu(f_{n+1})=0}$. Let ${1\leqslant q_1<\ldots <q_{n-m}\leqslant n}$ be the complementary collection to $\{p_1,\ldots,p_m\}$ in $\{1,\ldots,n\}$. Let $\sigma$ be the permutation that sends $(1,\ldots,n)$ to $(p_1,\ldots,p_m,q_1,\ldots,q_{n-m})$. Define an element $(M,\mu)\in\Kc$, where ${M:=\{q_1,\ldots,q_{n-m}\}}$ and $\mu(q_i):=m+i$, ${1\leqslant i\leqslant n-m}$. Then there is an equality
\begin{equation}\label{eq:Denis}
CC_n(f_1,f_2,\ldots,f_{n+1})=(-1)^{\sgn(\nu(f_1),\,\ldots,\,\nu(f_{n+1}))}\,
\prod_{K\subset M}\langle f_1,g_{(K,\kappa)}\rangle^{(-1)^{(|K|+\sgn(\sigma)+n-m)}}\,,
\end{equation}
where $K$ runs over all subsets of $M$ (including the empty set) and $\kappa$ is the restriction of $\mu$ to $K$. Indeed, in this case, there are equalities
$$
\alpha\big(e\cdot\nu(f_{i+1})+x_i\big)=\alpha(e_{p_i})\,,\qquad 1\leqslant i\leqslant m\,,
$$
$$
\alpha\big(e\cdot\nu(f_{i+1})+x_i\big)=\alpha(x_i)\,,\qquad m+1\leqslant i\leqslant n\,.
$$
Now we apply Remark~\ref{rmk:tilde} and formula~\eqref{eq:wedge}.

Note that if $m<n$, then $\sgn\big(\nu(f_1),\ldots,\nu(f_{n+1})\big)\equiv 0\pmod{2}$, because~${\nu(f_{n+1})=0}$.
\end{itemize}
\end{examp}

\begin{rmk}
Example~\ref{ex:tilde}(ii) has the following generalization. Suppose that the \mbox{$(n\times n)$-matrix} $\big(\nu(f_2),\ldots,\nu(f_{n+1})\big)$ is upper-triangular with units on the diagonal. Then there is an equality
$$
CC_n(f_1,f_2,\ldots,f_{n+1})=(-1)^{\sgn(\nu(f_1),\,\ldots,\,\nu(f_{n+1}))}\,\langle f_1,f_2,\ldots,f_{n+1}\rangle\,.
$$
Indeed, this follows from invariance of the map $\sgn$ under automorphisms of the \mbox{$\uz(A)$-mo\-dule}~$\uz^n(A)$, formula~\eqref{eq:invnu}, Corollary~\ref{cor:invCC}, and Example~\ref{ex:tilde}(ii) applied to the collection ${\phi^{-1}(f_1),t_1,\ldots,t_n\in\LL^n(A)^*}$, where $\phi\colon\LL^n(A)\to\LL^n(A)$ is a continuous automorphism such that $\phi(t_i)=f_{i+1}$, $1\leqslant i\leqslant n$.
\end{rmk}

\medskip

\begin{rmk}\label{rmk:explalt}
Here is a more explicit way to compute the higher-dimensional Contou-Carr\`ere symbol, essentially using Example~\ref{ex:tilde}(iii), which is based on
Theorem~\ref{theor:tilde} and Remark~\ref{rmk:tilde}. Actually, this  gives also a correct definition (slightly different from formula~\eqref{eq:CCtilde}) of the higher-dimensional Contou-Carr\`ere symbol over an arbitrary ring. Let~$A$ be any ring and let $f_1,\ldots,f_{n+1}\in\LL^n(A)^*$. In order to compute $CC_n(f_1,\ldots,f_{n+1})\in A^*$, one makes the following steps.

\begin{enumerate}
\item
Take the decomposition into a finite product of rings $A\simeq\prod\limits_{j=1}^N A_j$ such that for all~$j$ with $1\leqslant j\leqslant N$, the restriction of $(\nu(f_1), \ldots, \nu(f_{n+1}))$ to $\Spec(A_j)$ is a constant $\z^{n^2}$-valued function and
 for all $j_1, j_2$ with $1 \leqslant  j_1 < j_2 \leqslant N $, the restrictions of $(\nu(f_1), \ldots, \nu(f_{n+1}))$ to $\Spec(A_{j_1})$ and $\Spec(A_{j_2})$ are not equal.
 Clearly, this decomposition is unique.  By functoriality of $CC_n$  we have
$$
CC_n(f_1, \ldots , f_{n+1})= \prod_{j=1}^N CC_n(f_{1,j}, \ldots, f_{n+1, j})\,,
$$
where   $f_i = \prod\limits_{j=1}^N f_{i,j}$ for any $1 \leqslant i \leqslant n+1$, and $f_{i,j} \in \LL^n(A_j)^*$.
Replace $A$ by each of~$A_j$, thus suppose from now on that all $\nu(f_i)$ are constant $\z^{n}$-valued functions.
\item
By decomposition~\eqref{eq:dec1} and multilinearity of $CC_n$, the symbol $CC_n(f_1,\ldots,f_{n+1})$ decomposes uniquely into a product of symbols $CC_n(g_1,\ldots,g_{n+1})$ such that for each $i$, $1\leqslant i\leqslant n+1$, we have that either $g_i\in\{t_1,\ldots,t_n\}$, or $\nu(g_i)=0$. Replace the collection $(f_1,\ldots,f_{n+1})$ by each collection of such type.
\item
Suppose that there are equal elements among $f_1,\ldots,f_{n+1}$ and let $f_i=f_j$ be the equality with the smallest pair $(i,j)$, $i<j$, where we consider a natural lexicographical order on the pairs. Then put
$$
CC_n(f_1,\ldots,f_{n+1})=(-1)^{\,\det(\nu(f_1),\,\ldots,\,\widehat{\nu(f_j)},\ldots,\,\nu(f_{n+1}))}\,,
$$
where the notation $\widehat{\nu(f_j)}$ means that we omit $\nu(f_j)$. Note that here we use the antisymmetric property of $CC_n$ and formula~\eqref{eq:easy_St}, which is true for any ring $A$, see reasonings in Section~\ref{sec:prel}.
\item
Suppose that all elements $f_1,\ldots,f_{n+1}$ are different. Then there is a unique permutation $\sigma$ of the collection $(f_{2},\ldots, f_{n+1})$ such that $\sigma$ preserves the order between $f_i$'s with $\nu(f_i)=0$ and the result of the permutation is of type $(t_{p_1},\ldots,t_{p_m},g_{m+1},\ldots,g_n)$, where $0\leqslant m\leqslant n$, $1\leqslant p_1<\ldots<p_m\leqslant n$, and $\nu(g_i)=0$ for all $i$, $m+1\leqslant i\leqslant n$. Using the antisymmetric property of~$CC_n$, put
$$
CC_n(f_1,f_2,\ldots,f_{n+1})=CC_n(f_1,t_{p_1},\ldots,t_{p_m},g_{m+1},\ldots,g_n)^{\sgn(\sigma)}
$$
and replace the collection $(f_1,f_2,\ldots,f_{n+1})$ by the collection $(f_1,t_{p_1},\ldots,t_{p_m},g_{m+1},\ldots,g_n)$.
\item
For a collection $(f_1,f_2,\ldots,f_{n+1})=(f_1,t_{p_1},\ldots,t_{p_m},f_{m+2},\ldots,f_{n+1})$ such that ${\nu(f_{m+1})=\ldots=\nu(f_{n+1})=0}$, define $CC_n(f_1,f_2,\ldots,f_{n+1})$ by formula~\eqref{eq:Denis}.
\end{enumerate}
\end{rmk}

\begin{rmk} \label{rmk:integrab}
Take positive integers $1\leqslant j_1<\ldots<j_q\leqslant n$, where $1\leqslant q\leqslant n$. Let $\Q[[x_{i,l}]]$ be the ring of power series in formal variables $x_{i,l}$, where $1\leqslant i\leqslant n+1-p$ and $l\in\z^n$, see~\cite[Def.\,5.1]{GOMS}. Consider infinite (in all ''directions'') series
$\mbox{$f_i:=1+\sum\limits_{l\in\z^n}x_{i,l}t^l$}$, where $1\leqslant i\leqslant n+1-p$. Formally opening brackets in the expression (cf. formula~\eqref{eq:CC3})
$$
\exp\,\res\left(\log(f_1)\,\frac{df_2}{f_2}\wedge\ldots\wedge
\frac{df_{p}}{f_{p}}\wedge\frac{dt_{j_1}}{t_{j_1}}\wedge\ldots\wedge\frac{dt_{j_q}}{t_{j_q}}\right)\,,
$$
we obtain a power series $\varphi_{n,j_1,\ldots,j_q}\in \Q[[x_{i,l}]]$. A more rigorous definition of the series~$\varphi_{n,j_1,\ldots,j_q}$ is given in~\cite[\S\,8.6]{GOMS}, where it is also shown that, in fact, this series has integral coefficients. The proof of this fact uses the existence of the higher-dimensional Contou-Carr\`ere symbol for all rings, which was constructed in~\cite{GOMS} with the help of the boundary map for algebraic $K$-groups by formula~\eqref{eq:K}, see the discussion before~\cite[Ex.\,8.32]{GOMS}. Now, Theorem~\ref{theor:tilde} gives an explicit formula for the higher-dimensional Contou-Carr\`ere symbol for any ring. After all, this gives a new proof of the integrality of coefficients of the series $\varphi_{n,j_1,\ldots,j_q}$, without the use of algebraic $K$-theory.
\end{rmk}

\bigskip

Sergey Gorchinskiy

Steklov Mathematical Institute of Russian Academy of Sciences, ul. Gubkina 8, Moscow, 119991 Russia

{\em E-mail address}: gorchins@mi.ras.ru

\bigskip

Denis Osipov

Steklov Mathematical Institute of Russian Academy of Sciences, ul. Gubkina 8, Moscow, 119991 Russia

National University of Science and Technology MISIS, Leninskii pr. 4, Moscow, 119049 Russia

{\em E-mail address}: d\_osipov@mi.ras.ru

\end{document}